\documentclass[12pt,reqno]{amsart}

\usepackage[arrow, matrix, curve]{xy} 

\usepackage{tikz, pgfplots}
\usepackage{tikz-cd} 

\usepackage{amssymb,hyperref}

\newtheorem{theorem}{Theorem}[section]
\newtheorem{proposition}[theorem]{Proposition}
\newtheorem{lemma}[theorem]{Lemma}
\newtheorem{corollary}[theorem]{Corollary}

\theoremstyle{definition}
\newtheorem{definition}[theorem]{Definition}
\newtheorem{remark}[theorem]{Remark}

\numberwithin{equation}{section}

\begin{document}

\baselineskip=15pt

\title[Branes and Higgs bundles on projective varieties]{Branes and moduli spaces of Higgs 
bundles on smooth projective varieties}

\author[I. Biswas]{Indranil Biswas}

\address{School of Mathematics, Tata Institute of Fundamental
Research, Homi Bhabha Road, Mumbai 400005, India}

\email{indranil@math.tifr.res.in}

\author[S. Heller]{Sebastian Heller}

\address{Institute of Differential Geometry,
Leibniz Universit\"at Hannover,
Welfengarten 1, 30167 Hannover}

\email{seb.heller@gmail.com}

\author[L. P. Schaposnik]{Laura P. Schaposnik}

\address{Department of Mathematics, Statistics, and Computer Science, University of
Illinois at Chicago, 851 S Morgan St, Chicago, IL 60607, United States}

\email{schapos@uic.edu}



\thanks{IB: Corresponding author; indranil@math.tifr.res.in}

\date{}

\begin{abstract}
Given a smooth complex projective variety $M$ and a smooth closed curve $X\, \subset\, M$ such that
the homomorphism of fundamental groups $\pi_1(X)\, \longrightarrow\, \pi_1(M)$ is surjective,
we study the restriction map of Higgs bundles, namely from the Higgs bundles on $M$ to those on $X$.
In particular, we investigate the interplay between this restriction map and various types of
branes contained in the moduli spaces of Higgs bundles on $M$ and $X$. We also consider
the set-up where a finite group is acting on $M$ via holomorphic automorphisms or
anti-holomorphic involutions, and the curve
$X$ is preserved by this action. Branes are studied in this context.
\end{abstract}

\maketitle

\section{Introduction}\label{intro.}

Lagrangian and holomorphic spaces have been of much interest within the study of Higgs bundles
both on Riemann surfaces and on higher dimensional varieties. We shall consider here
an irreducible smooth complex
projective variety $M$ of dimension $d$, and consider the following moduli spaces:
\begin{itemize}
\item the Betti moduli 
space ${\mathcal B}_M(r)$ of equivalence classes of representations of $\pi_1(M,\, x_0)$
in $\text{GL}(r, {\mathbb C})$;

\item the moduli space ${\mathcal H}_M(r)$ of semistable Higgs bundles on $M$, of rank $r$ and vanishing
Chern classes; and

\item the Deligne--Hitchin moduli space $\mathcal M^{DH}_M(r)$ which is the twistor space of
the hyperK\"ahler space $\mathcal H_M(r)$.
\end{itemize}

We shall dedicate this paper to the study of Lagrangian and holomorphic spaces appearing within the above moduli 
spaces arising through the inclusion of curves and hypersurfaces in $M$, as well as the action of 
anti-holomorphic involutions of $M$ and also the finite group actions on $M$. In the case of Riemann surfaces, real 
structures have long been studied in relation to Hitchin systems. In particular, \cite{aba,slices} initiated the 
study of branes arising through anti-holomorphic involutions from the perspective of \cite{kap}. Lagrangians 
within moduli spaces of Higgs bundles arising through finite group actions were considered also in the setting of Riemann surfaces 
in \cite{finite}. In the present manuscript, we shall follow the lines of thought of the above papers.
 
After introducing Higgs bundles from a differential geometric perspective, the Betti moduli space of 
representations, and a brief description of nonabelian Hodge theory in Section \ref{nah}, we study the 
implications of considering a smooth closed curve $X$ on $M$ such that the homomorphism of fundamental groups 
$\pi_1(X,\, x_0)\, \longrightarrow\, \pi_1(M,\, x_0)$ induced by the inclusion map $X\, \hookrightarrow\, M$ is 
surjective (this is done in Section \ref{Section:Restriction_to_curves}).
This inclusion map induces maps $$\Phi\, :\, {\mathcal 
B}_M(r)\, \longrightarrow\, {\mathcal B}_X(r)$$ on the Betti moduli spaces and $$\Psi\, :\, {\mathcal H}_M(r)\, 
\longrightarrow\, {\mathcal H}_X(r)$$ on the Higgs moduli spaces, as seen in \eqref{e5} and \eqref{e7} 
respectively. The given condition that the homomorphism $\pi_1(X,\, x_0)\, \longrightarrow\, \pi_1(M,\, x_0)$
induced by the inclusion map $X\, \hookrightarrow\, M$ is
surjective ensures that the above maps $\Phi$ and $\Psi$ are actually embeddings.
By studying these maps, through Theorem \ref{thm1} and Corollary \ref{cor1} we are able to show 
the following.

\smallskip
\noindent {\bf Theorem A.}\, {\it The above map $\Psi$ makes ${\mathcal H}_M(r)$ a hyperK\"ahler subspace of
${\mathcal H}_X(r)$, and thus the subspace ${\mathcal H}_M(r)$ is a 
$(B,\,B,\,B)$--brane in ${\mathcal H}_X(r)$.} 
\smallskip

The case of Deligne--Hitchin moduli spaces is considered in Section \ref{Del}, where we look into the relationship
of these spaces with twistor spaces of certain hyperK\"ahler manifolds. Let ${\mathcal M}^{DH}_M(r)$ and
${\mathcal M}^{DH}_X(r)$ be the Deligne--Hitchin moduli spaces for $M$ and $X$ respectively.
In this setting, through an induced map $$ 
\Upsilon\, :\, {\mathcal M}^{DH}_M(r)\, \longrightarrow\, {\mathcal M}^{DH}_X(r)$$
constructed using restriction to $X$ (see \eqref{etwistz}) we prove the following in Theorem \ref{thm2}:

\smallskip
\noindent {\bf Theorem B.}\,
{\it The above map $\Upsilon$ makes ${\mathcal M}^{DH}_M(r)$ a twistor subspace of
${\mathcal M}^{DH}_X(r)$.}
\smallskip

We begin the study of actions by considering real structures $\sigma\,:\,M\,\longrightarrow\, M$ in Section 
\ref{sec:real}. By observing that one may choose the real structure and the subvariety $X$ such that they are 
compatible (Lemma \ref{lem12}), we show in Theorem \ref{prop1} the following.

\smallskip
\noindent {\bf Theorem C.}\, {\it Choose the closed curve $X$ such that it is
preserved by the real structure $\sigma$ on $M$. Take $S$ to be $X$ or $M$. Then, the $C^\infty$ involution
${\mathcal I}_S$ induced on $\mathcal B_S(r)$ by $\sigma$ (the restriction of $\sigma$ when $S\,=\, X$)
is anti-holomorphic with respect to the
complex structure $I$ (that gives the complex structure of the moduli space
of Higgs bundles ${\mathcal H}_S(r)$) and holomorphic with respect to the
complex structure $J$ (that gives the complex structure of
${\mathcal B}_S(r)$), and thus the fixed point set ${\mathcal H}_S(r)^{{\mathcal I}_S}$ is a
$(A,\,B,\,A)$--brane in the corresponding moduli space ${\mathcal H}_S(r)$. Furthermore,
$${\mathcal I}_X\circ\Psi\,=\, \Psi\circ{\mathcal I}_M$$ for $\Psi$ in Theorem A.}
\smallskip

The case of finite group actions on $M$ is studied in Section \ref{sec:finite}. This line of research begun in 
\cite{finite} for rank two Higgs bundles on Riemann surfaces where it was shown that equivariant Higgs bundles under a 
finite group give natural $(B,\,B,\,B)$--branes in the moduli space of Higgs bundles. Here, we shall address the 
higher rank cases and consider finite group actions on the moduli space of fixed determinant Higgs bundles 
for which the trace of the Higgs field being zero (${\rm SL}(r,\mathbb{C})$--Higgs bundles). One of our main results, 
appearing in Theorem \ref{teo1} and Proposition \ref{teo2} is the following:

\smallskip
\noindent {\bf Theorem D.}\, {\it Let $M$ be a compact Riemann surface of genus $g$
on which a finite group $\Gamma$ is acting via holomorphic automorphisms.
Consider the moduli space 
$\mathcal{M}_{{\rm SL}(r,\mathbb{C})}\, \subset\, {\mathcal H}_M(r)$ of semistable Higgs bundles
on $M$ whose structure group is ${\rm SL}(r,\mathbb{C})$. Then, for $r\,>\,2$ the following hold:
\begin{itemize}
\item For even genus $g$, the fixed point locus in $\mathcal{M}_{{\rm SL}(r,\mathbb{C})}$ giving a $(B,\,B,\,B)$--brane
will never be a mid dimensional space. 

\item For odd genus $g$, the fixed point free action of $\Gamma\,:=\,\mathbb{Z}/2\mathbb{Z}$ on
$M$ defines a mid dimensional space as its fixed point locus on $\mathcal{M}_{{\rm SL}(r,\mathbb{C})}$.
\end{itemize}}
\smallskip

The subspaces constructed in Theorem D can be further studied, and as shown in Proposition \ref{teo3}, 
all these mid-dimensional spaces constructed in Theorem \ref{teo1} and Proposition \ref{teo2} are
in fact $(B,\,B,\,B)$--branes.

\section{Nonabelian Hodge theory}\label{nah}

In what follows we shall introduce Higgs bundles from a differential geometric perspective, and also
the Betti moduli space of representations, and then give a brief description of nonabelian Hodge theory. 

\subsection{Higgs bundles}

Let $M$ be an irreducible smooth complex projective variety of dimension $d$. The holomorphic cotangent and
tangent bundles of $M$ will be denoted by $\Omega^1_M$ and $TM$ respectively. The $i$--fold exterior
product $\bigwedge^i \Omega^1_M$ will be denoted by $\Omega^i_M$.
Fixing a very ample line bundle $L$ on $M$, the degree of a torsionfree coherent sheaf $F$ on $X$
is defined to be
$$
\text{degree}(F)\, :=\, (c_1(\det F)\cup c_1(L)^{d-1})\cap [M] \,\in\, \mathbb Z
$$
(see \cite[Ch.~V,\S~6]{Ko} for the determinant bundle $\det F$). The number
$$
\mu(F)\, :=\, \frac{\text{degree}(F)}{\text{rank}(F)}\, \in\, \mathbb Q
$$
is called the \textit{slope} of $F$. From \cite{Hi1,Hi2,Si1,Si2} one has the following definition:

\begin{definition}A \textit{Higgs field} on a vector bundle $E$ over $M$ is a holomorphic section
$$
\theta\, \in\, H^0(M,\, \text{End}(E)\otimes\Omega^1_M)
$$
such that the section $\theta\bigwedge\theta$ of $\text{End}(E)\otimes\Omega^2_M$ vanishes
identically. A \textit{Higgs bundle} on $M$ is a pair $(E,\, \theta)$,
where $E$ is a holomorphic vector bundle and $\theta$ is a Higgs field on $E$.
\end{definition}

\begin{definition}A Higgs bundle $(E,\, \theta)$ is called \textit{stable}
(respectively, \textit{semistable}) if
$$
\mu(F)\, <\, \mu(E) \ \ \ \text{(respectively,}~ \mu(F)\, \leq\, \mu(E)\text{)}
$$
for all coherent subsheaves $F\, \subset\, E$ with $0\, <\, \text{rank}(F)\, <\, \text{rank}(E)$
and $\theta(F)\, \subset\, F\otimes\Omega^1_M$ \cite{Hi1,Si1,Si2}. A Higgs bundle
$(E,\, \theta)$ is called \textit{polystable} if
\begin{itemize}
\item $(E,\, \theta)$ is semistable, and

\item $(E,\, \theta)\,=\, \bigoplus_{i=1}^\ell (E_i,\, \theta_i)$, where each $(E_i,\, \theta_i)$
is a stable Higgs bundle.
\end{itemize}
\end{definition}

Let $(E,\, \theta)$ be a Higgs bundle on $M$ such that
\begin{equation}\label{e2}
\text{degree}(E)\,=\, 0\, =\, (ch_2(E)\cup c_1(L)^{d-2})\cap [M]\, ;
\end{equation}
we recall that $ch_2(E)\,=\, \frac{1}{2}c_1(E)^2 - c_2(E)$.
Given a Hermitian structure $h$ on $E$, the Chern connection on $E$ corresponding to $h$ will be
denoted by $\nabla^h$, and the curvature of the connection $\nabla^h$ will be denoted by ${\mathcal K}(\nabla^h)$.
Let
$$
\theta^*\, \in\, C^\infty(M;\, \text{End}(E)\otimes\Omega^{0,1}_M)
$$
be the adjoint of $\theta$ with respect to the Hermitian structure
$h$. Then, from \cite{Hi1,Si1,Si2}, the following is the Hermitian--Yang--Mills equation:
\begin{equation}\label{e3}
{\mathcal K}(\nabla^h) +[\theta\,,\theta^*]\, =\, 0.
\end{equation}

A theorem of Simpson says that $E$ admits a Hermitian metric satisfying the 
Hermitian--Yang--Mills equation for $(E,\, \theta)$ if and only if $(E,\, \theta)$ is a polystable 
Higgs bundle \cite[Theorem 1]{Si1}, \cite[Proposition 3.4]{Si1}, \cite[Theorem 1]{Si2}. When $E$ is a rank 
two vector bundle on a smooth complex projective curve, the result was proven earlier by 
Hitchin in \cite{Hi1}.

We note that the Hermitian--Yang--Mills equation in \eqref{e3} implies that
\begin{equation}\label{e4}
\nabla^h+\theta+\theta^*
\end{equation}
is a flat connection on $E$. Let
$$
\nabla^h\,=\, (\nabla^h)^{1,0}+(\nabla^h)^{0,1}
$$
be the decomposition of the Chern connection $\nabla^h$ into $(1,\, 0)$ and $(0,\, 1)$ components. 

\begin{remark}The holomorphic structure on $E$ is given by the Dolbeault operator $(\nabla^h)^{0,1}$. Note that 
the Dolbeault operator for the holomorphic structure on the $C^\infty$ bundle $E$ given by 
the flat connection in \eqref{e4} is $(\nabla^h)^{0,1}+\theta^*$. Therefore, the holomorphic 
vector bundle given by the flat connection corresponding to a Higgs bundle does not 
coincide, in general, with the holomorphic vector bundle underlying the Higgs bundle.
\end{remark}
\subsection{Flat connections}

Given a base point $x_0\, \in\, M$, a representation
$$
\rho\, :\, \pi_1(M,\, x_0)\, \longrightarrow\, \text{GL}(r,{\mathbb C})
$$
is called {\it irreducible} if the standard action of $\rho(\pi_1(M,\, x_0))\, \subset\,
\text{GL}(r,{\mathbb C})$ on ${\mathbb C}^r$ does not preserve any nonzero proper subspace of
${\mathbb C}^r$. The homomorphism $\rho$ is called {\it completely reducible} if it is a direct sum
of irreducible representations.
Two homomorphisms $$\rho_1,\, \rho_2\, :\, \pi_1(M,\, x_0)\, \longrightarrow\, \text{GL}(r,{\mathbb C})$$
are called {\it equivalent} if there is an element $g\, \in\, \text{GL}(r,{\mathbb C})$ such
that
$$
\rho_1(\gamma)\,=\, g^{-1} \rho_2(\gamma)g
$$
for all $\gamma\, \in\, \pi_1(M,\, x_0)$. Clearly, this equivalence relation preserves both irreducibility
and complete reducibility. The space of equivalence classes of completely reducible
homomorphisms from $\pi_1(M,\, x_0)$ to $\text{GL}(r,{\mathbb C})$ has the structure of an affine
scheme defined over $\mathbb C$, which can be seen as follows. Note that $\pi_1(M,\, x_0)$ is a finitely
presented group and $\text{GL}(r,{\mathbb C})$ is a complex affine algebraic group.
Therefore, the space of all homomorphisms
$\text{Hom}(\pi_1(M,\, x_0),\, \text{GL}(r,{\mathbb C}))$ is a complex affine scheme.
The adjoint action of $\text{GL}(r,{\mathbb C})$ on itself produces
an action of $\text{GL}(r,{\mathbb C})$ on $\text{Hom}(\pi_1(M,\, x_0),\, \text{GL}(r,{\mathbb C}))$. The geometric
invariant theoretic quotient $$\text{Hom}(\pi_1(M,\, x_0),\, \text{GL}(r,{\mathbb C}))/\!\!/\text{GL}(r,{\mathbb C})$$
is the moduli space of equivalence classes of completely reducible
homomorphisms from $\pi_1(M,\, x_0)$ to $\text{GL}(r,{\mathbb C})$ \cite{Si3,Si4}.

We shall let ${\mathcal B}_M(r)$ denote this moduli space of equivalence classes of completely reducible
homomorphisms from $\pi_1(M,\, x_0)$ to $\text{GL}(r,{\mathbb C})$, which is known as the
Betti moduli space.

A homomorphism $\rho\, :\, \pi_1(M,\, x_0)\, \longrightarrow\, \text{GL}(r,{\mathbb C})$ 
produces an algebraic vector bundle $E$ on $M$ of rank $r$ equipped with a flat (i.e., 
integrable) algebraic connection, together with an isomorphism of the fiber $E_{x_0}$ with 
${\mathbb C}^r$. Equivalence classes of such homomorphisms correspond to algebraic vector 
bundles of rank $r$ equipped with a flat algebraic connection; this is an example of the
Riemann--Hilbert correspondence.

\begin{definition}A connection $\nabla$ on a vector bundle $E$ is called \textit{irreducible} 
if there is no subbundle $0\, \not=\, F\, \subsetneq\, E$ which is preserved by $\nabla$. A 
connection $\nabla$ on a vector bundle $E$ is called \textit{completely reducible} if
$$
(E,\,\nabla)\,=\, \bigoplus_{i=1}^N(E_i,\,\nabla^i)\, ,
$$
where each $\nabla^i$ is an irreducible connection on $E_i$. \end{definition}

We note that irreducible 
(respectively, completely reducible) flat algebraic connections of rank $r$ on $M$ 
correspond to irreducible (respectively, completely reducible) equivalence classes of 
homomorphisms from $\pi_1(M,\, x_0)$ to $\text{GL}(r,{\mathbb C})$.

\subsection{Harmonic structures}

Let $(E,\,\nabla)$ be a flat vector bundle of rank $r$ on $M$, and let
$$
\rho\, :\, \pi_1(M,\, x_0)\, \longrightarrow\, \text{GL}(E_{x_0})
$$
be the monodromy homomorphism for $\nabla$.

Given the universal cover $\varpi\, :\, (\widetilde{M},\, \widetilde{x}_0)\, \longrightarrow\,
(M,\, x_0)$, using the pulled back connection $\varpi^*\nabla$, the pulled back bundle $\varpi^*E$
is identified with the trivial bundle $\widetilde{M}\times E_{x_0}\, \longrightarrow\,
\widetilde{M}$. Therefore, a
Hermitian structure $h$ on $E$ gives a $C^\infty$ map
\begin{equation}\label{fh}
F^h\, :\, \widetilde{M}\, \longrightarrow\, \text{GL}(E_{x_0})/\text{U}(E_{x_0})\, ,
\end{equation}
where $\text{U}(E_{x_0})$ consists of all automorphisms of $E_{x_0}$ that preserve
the Hermitian structure $h(x_0)$ on it. The group
$\pi_1(M, \, x_0)$ acts on $\text{GL}(E_{x_0})/\text{U}(E_{x_0})$ through the left translation action of
$\rho(\pi_1(M,\, x_0))$. We note that the map $F^h$ in \eqref{fh} is $\pi_1(M, \, x_0)$--equivariant. The
quotient space $\text{GL}(E_{x_0})/\text{U}(E_{x_0})$ is equipped with a Riemannian metric; it is given by
the restriction of the Killing form on the orthogonal complement of
${\rm Lie}(\text{U}(E_{x_0}))\, \subset\, {\rm Lie}(\text{GL}(E_{x_0}))$.

The Hermitian structure $h$ is called \textit{harmonic} if the map $F^h$ in \eqref{fh} is harmonic 
with respect to a K\"ahler structure on $M$. It should be emphasized that the harmonicity 
condition for $F^h$ does not depend on the choice of the K\"ahler form on $M$.

A theorem of Corlette says that $(E,\,\nabla)$ admits a harmonic Hermitian structure if and 
only if $\nabla$ is completely reducible \cite{Co} (when $E$ is of rank two on a smooth 
complex projective curve, this was proved in \cite{Do} by Donaldson).

Let $(E,\, \theta)$ be a polystable Higgs bundle on $M$ satisfying \eqref{e2}, and consider a 
Hermitian structure $h$ on $E$ satisfying the Hermitian--Yang--Mills equation in \eqref{e3}. 
Then, the flat connection $\nabla^h+\theta+\theta^*$ on $E$ in \eqref{e4} is completely 
reducible, and $h$ is a harmonic Hermitian structure for the flat connection 
$\nabla^h+\theta+\theta^*$. Conversely, if $h$ is a harmonic Hermitian structure for a flat 
connection $(E,\, \nabla)$, then $h$ and $\nabla$ together define a holomorphic structure on 
$E$ and a Higgs field $\theta$ on $E$ for this holomorphic structure such that $h$ satisfies 
the Hermitian--Yang--Mills equation for the Higgs bundle $(E,\, \theta)$.

As shown in \cite[p.~20, Corollary 1.3]{Si2}, the above constructions produce an equivalence 
of categories between the following two categories:
\begin{enumerate}
\item Objects are completely reducible flat algebraic connections on $M$, and morphisms are connection
preserving homomorphisms.

\item Objects are polystable Higgs bundles $(E,\, \theta)$ on $M$
satisfying \eqref{e2}; the morphisms are homomorphisms of Higgs bundles.
\end{enumerate}

We also note that if $(E,\, \theta)$ is
a polystable Higgs bundle on $M$ satisfying \eqref{e2}, then all the rational Chern classes of
$E$ of positive degree vanish \cite[p.~878--879, Proposition 3.4]{Si1}. (See also \cite{We} for an
exposition for dimension one.)

\section{Restriction to curves}\label{Section:Restriction_to_curves}

We shall dedicate this section to understanding the restriction of the ideas of Section 
\ref{intro.} to hypersurfaces, which will become useful when studying Lagrangians within the 
moduli space of Higgs bundles.

\subsection{Higgs bundles and hypersurfaces}\label{restriction}

The homotopical version of the Lefschetz hyperplane theorem says that for a smooth very ample 
hypersurface $H\, \subset\, M$, the homomorphism of fundamental groups $\pi_1(H,\, x_0)\, 
\longrightarrow\, \pi_1(M,\, x_0)$ induced by the inclusion map $H\, \hookrightarrow\, M$, 
where $x_0\, \in\, H$, is surjective, and it is an isomorphism if $d\, =\, \dim M \, \geq\, 3$ (see 
\cite[p.~48, (8.1.1)]{La}). Consequently, using induction on the number of hypersurfaces we 
conclude that for a smooth closed curve $X$ on $M$ which is the intersection of $d-1$ smooth 
very ample hypersurfaces on $M$, the homomorphism of fundamental groups $\pi_1(X,\, x_0)\, 
\longrightarrow\, \pi_1(M,\, x_0)$ induced by the inclusion map $X\, \hookrightarrow\, M$ is 
surjective.

Fix a smooth closed curve $X$ on $M$ such that the homomorphism of fundamental groups
$\pi_1(X,\, x_0)\, \longrightarrow\, \pi_1(M,\, x_0)$ induced by the inclusion map $X\, \hookrightarrow\, M$
is surjective, and let
\begin{equation}\label{e1}
\phi\, :\, X\, \hookrightarrow\, M
\end{equation}
be the inclusion map.

As done previously, consider a polystable Higgs bundle $(E,\, \theta)$ on $M$ satisfying 
\eqref{e2}, and let $h$ be a Hermitian structure on $E$ that satisfies the 
Hermitian--Yang--Mills equation in \eqref{e3} for the Higgs bundle $(E,\, \theta)$.

\begin{lemma}\label{3.1}
The Hermitian structure $\phi^*h$ on $\phi^*E$ satisfies Hermitian--Yang--Mills
equation for the induced Higgs bundle $(\phi^*E,\, \phi^*\theta)$ on $X$.
\end{lemma}

\begin{proof}
The pulled back section $\phi^*\theta$ defines a Higgs field on $\phi^*E$ using the
natural homomorphism
$$
(d\phi)^*\, :\, \phi^*\Omega^1_M\, \longrightarrow\, \Omega^1_X\, ,
$$
where $d\phi\, :\, TX\, \longrightarrow\, \phi^*TM$ is the differential of the map $\phi$ in \eqref{e1}.
The vector bundle $\phi^*E$ has the pulled back Hermitian 
structure $\phi^*h$. Note that the Chern connection $\nabla^{\phi^*h}$ on $\phi^*E$ for the 
Hermitian structure $\phi^*h$ coincides with $\phi^*\nabla^h$, where $\nabla^h$ is the Chern 
connection on $E$ for the Hermitian structure $h$. Moreover, we also have that 
$\phi^*(\theta^*)\,=\, (\phi^*\theta)^*$. Using these, and the fact that $h$ satisfies the 
Hermitian--Yang--Mills equation for the Higgs bundle $(E,\, \theta)$, the lemma follows. 
\end{proof}

From Lemma \ref{3.1} one has that
\begin{itemize}
\item the Higgs bundle $(\phi^*E,\, \phi^*\theta)$ is polystable, and

\item $\nabla^{\phi^*h}+\phi^*\theta + \phi^*\theta^*$ is a flat connection on
$\phi^*E$.
\end{itemize}
In particular, we have $\text{degree}(\phi^*E)\,=\, 0$.
Let $(W,\, \nabla)$ be the flat connection on $M$ corresponding to the Higgs
bundle $(E,\, \theta)$; recall that $E$ and $W$ are the same
$C^\infty$ vector bundle with possibly different holomorphic structures. From Lemma \ref{3.1}
it follows that the flat connection
corresponding to $(\phi^*E,\, \phi^*\theta)$ coincides with
the pulled back flat connection $(\phi^*W,\, \phi^*\nabla)$.

\subsection{Harmonic structures and hypersurfaces}

Consider now an algebraic vector bundle $V$ of rank $r$ on $M$ equipped with a flat algebraic 
completely reducible connection $\nabla$. For $x_0\, \in\, X$ and the map $\phi$ in 
\eqref{e1}, the homomorphism
\begin{equation}\label{e6}
\phi_*\, :\, \pi_1(X,\, x_0)\,\longrightarrow\, \pi_1(M,\, \phi(x_0))\, ,
\end{equation}
is surjective, and thus it follows
immediately that $\phi^*\nabla$ is a flat algebraic
completely reducible connection on $\phi^*V$.

Let $h_V$ be a harmonic Hermitian
structure on $V$ for $\nabla$, and let $(W,\, \theta)$
be the Higgs bundle corresponding to $(V,\, \nabla)$. Recall that $V$ and $W$ are the same
$C^\infty$ vector bundle with possibly different holomorphic structures. It is straightforward to
check that the Hermitian structure $\phi^*h$ on $\phi^*V$ is harmonic for the
connection $\phi^*\nabla$. Indeed, this follows from the facts that $M$ is K\"ahler and
the embedding $\phi$ in \eqref{e1} is holomorphic. Consequently, the Higgs bundle for
$(\phi^*V,\, \phi^*\nabla)$ coincides with $(\phi^*W,\, \phi^*\theta)$.

As before, let ${\mathcal B}_X(r)$ (respectively, ${\mathcal B}_M(r)$) be the Betti moduli 
space of equivalence classes of completely reducible representations of $\pi_1(X,\, x_0)$ (respectively, 
$\pi_1(M,\, \phi(x_0))$ in $\text{GL}(r, {\mathbb C})$. Let
\begin{equation}\label{e5}
\Phi\, :\, {\mathcal B}_M(r)\, \longrightarrow\, {\mathcal B}_X(r)
\end{equation}
be the map induced by $\phi_*$ in \eqref{e6}, which clearly is an algebraic map. We note that
the map $\Phi$ is injective, because the homomorphism $\phi_*$ is surjective.

Let ${\mathcal H}_X(r)$ be the moduli space of semistable Higgs bundles on $X$ of rank $r$ 
and degree zero; this moduli space was constructed in \cite{Si3}, \cite{Si4} (see \cite{CN} 
for its properties). Let ${\mathcal H}_M(r)$ be the moduli space of semistable Higgs bundles 
$(E,\, \theta)$ on $M$ of rank $r$ such that all the rational Chern classes of $E$ of 
positive degree vanish. It may be recalled that any semistable Higgs bundle $(E,\, \theta)$
satisfying \eqref{e2} has the property that all the rational Chern classes of
$E$ of positive degree vanish \cite[p.~39, Theorem 2]{Si2}. In view of Lemma \ref{3.1}, we have a map
\begin{equation}\label{e7}
\Psi\, :\, {\mathcal H}_M(r)\, \longrightarrow\, {\mathcal H}_X(r)
\end{equation}
defined by $(E,\, \theta)\, \longmapsto\, (\phi^*E,\, \phi^*\theta)$.
It is clearly an algebraic map. Indeed, the restriction map from an algebraic family of Higgs bundles on a variety to a
family of Higgs bundles on a subvariety is evidently algebraic.
As observed above, using the $C^\infty$ identification between ${\mathcal B}_X(r)$ and 
${\mathcal H}_X(r)$, and also between ${\mathcal B}_M(r)$ and ${\mathcal H}_M(r)$, the map 
$\Phi$ in \eqref{e5} coincides with the map $\Psi$ in \eqref{e7}.

As noted before, the map $\Phi$ in \eqref{e5} is injective, because $\phi_*$ is surjective. Therefore, from the 
above observation that $\Phi$ coincides with $\Psi$ using the $C^\infty$ identification between ${\mathcal B}_X(r)$ 
and ${\mathcal H}_X(r)$, and also between ${\mathcal B}_M(r)$ and ${\mathcal H}_M(r)$, we conclude that $\Psi$ is an 
embedding. As mentioned above, the map $\Phi$ is algebraic.

\subsection{HyperK\"ahler structure}

Recall from \cite{Hi1,Hi2} that the
moduli space ${\mathcal H}_X(r)$ has a natural hyperK\"ahler structure, and thus we may fix three 
complex structures $I,\,J$ and $K$ on ${\mathcal H}_X(r)$ satisfying the quaternionic equations. We shall let the complex structures $I$ and $J$ be 
the complex structures of ${\mathcal H}_X(r)$ and ${\mathcal B}_X(r)$ respectively; so $K$ is determined by the 
equation $$IJ\,=\, K\, .$$ Therefore, from the holomorphicity of the maps $\Phi$ and $\Psi$ in \eqref{e5} and \eqref{e7} 
we obtain the following:

\begin{theorem}\label{thm1}
The map $\Psi$ in \eqref{e7} makes ${\mathcal H}_M(r)$ a hyperK\"ahler subspace of
${\mathcal H}_X(r)$.
\end{theorem}

\begin{remark}\label{rem1}
Consider the holomorphic symplectic form $\sigma_I$ on ${\mathcal H}_X(r)$. This
symplectic manifold $({\mathcal H}_X(r),\, \sigma_I)$ has a natural algebraically
completely integrable structure given by the Hitchin map. Although
$\Psi^*\sigma_I$ is a holomorphic symplectic structure on ${\mathcal H}_M(r)$, the
map $\Psi$ is not Poisson because it is an embedding and $\dim {\mathcal H}_X(r)
\,>\, \dim {\mathcal H}_M(r)$ in general.
\end{remark}

Given the hyperK\"ahler spaces ${\mathcal H}_M(r)$ and ${\mathcal H}_X(r)$, and the three fixed
complex structures $(I,J,K)$ as before, we shall denote by $\omega_I, \omega_J,$ and $\omega_K$ the corresponding K\"ahler forms defined by
 \[\omega_I(X,\,Y)\,=\,g(IX,\,Y), ~{~}~\omega_J(X,\,Y)\,=\,g(JX,\,Y),~{~}~\omega_K(X,\,Y)
\,=\, g(KX,\,Y)\, ,\]
where $g$ is the Riemannian metric for the hyperK\"ahler structure.
Then, the induced complex symplectic forms $\Omega_I, \Omega_J,$ and $\Omega_K$ are given by
\begin{eqnarray}\Omega_I \,=\,
\omega_J+\sqrt{-1}\omega_K, ~{~}~\Omega_J =\omega_K+\sqrt{-1}\omega_I,~{~}~\Omega_K=
\omega_I+\sqrt{-1}\omega_J.
\label{symplectic}
\end{eqnarray}

Therefore, one may consider subspaces of ${\mathcal H}_M(r)$ and ${\mathcal H}_X(r)$ which are 
Lagrangian with respect to one of the symplectic structures in \eqref{symplectic}, or 
holomorphic with respect to one of the fixed complex structures $I$, $J$ or $K$. Following \cite{kap}, we 
shall say that a Lagrangian subspace with respect to a symplectic structure is an $A$--brane, 
and a holomorphic subspace with respect to a complex structure is a $B$--brane. Hence, with 
respect to $(I,J,K)$ and $(\omega_I,\omega_J,\omega_K)$ one may have branes given by triples 
of letters. In particular, form Theorem \ref{thm1} we have the following:

\begin{corollary}\label{cor1}
The subspace ${\mathcal H}_M(r)$ is a 
$(B,\,B,\,B)$--brane in ${\mathcal H}_X(r)$.
\end{corollary}

\section{Deligne--Hitchin moduli space and twistor space}\label{Del}

In what follows we shall first recall the notion of $\lambda$--connections and the 
Deligne--Hitchin moduli space (for details and proofs see Simpson \cite{Si5}). Then, we shall 
look into the relationship of these spaces with the twistor spaces of certain hyperK\"ahler 
manifolds.

\subsection{Deligne--Hitchin moduli spaces}

Let $M$ be an irreducible smooth complex projective variety of dimension $d$. We fix a 
$C^\infty$ complex vector bundle $V\,\longrightarrow\, M$ such that all its rational Chern 
classes of positive degree vanish. A $\lambda$--connection on $V\,\longrightarrow\, M$ is a 
triple $(\lambda,\, \overline{\partial},\, D)$ consisting of a complex number $\lambda$, a 
holomorphic structure $\overline\partial$ on $M$, i.e., $\overline\partial^2\,=\,0$, and a 
holomorphic first order differential operator $D$ satisfying the following two conditions:
\begin{itemize}
\item $D^2\,=\,0$, and

\item $D(fs)\,=\,\lambda(\partial f )s+f Ds$ (similar to the Leibniz rule), where $f$ is any locally defined
holomorphic function and $s$ is any locally defined holomorphic section of $V$.
\end{itemize}

\begin{remark}A $\lambda$--connection for $\lambda\,=\,0$ is a Higgs bundle
$(\overline\partial ,\, D)$, and
for $\lambda\,=\,1$ it is a flat holomorphic connection $D+\overline\partial$.
\end{remark}

Consider the moduli space $\mathcal M^{Hod}_M(r)$ of completely reducible $\lambda$--connections
(to clarify, $\lambda$ is not fixed), and let
$$
f_M\, :\, \mathcal M^{Hod}_M(r)\,\longrightarrow\,\mathbb C\, ,\ \
(\lambda,\, \overline{\partial},\, D)\, \longmapsto\, \lambda
$$
be the fibration. The analogous construction on the complex
conjugate space $\overline M$ (the $C^\infty$ manifold underlying $\overline M$ is $M$ itself while
the complex structure on it is $-J_M$, where $J_M$ is the complex structure on $M$) gives a fibration
$$
f_{\overline{M}}\, :\, \mathcal M^{Hod}_{\overline{M}}(r)\,\longrightarrow\,\mathbb C\, .
$$
Moreover, there is the following gluing constructed by Deligne:
\begin{eqnarray}\varphi\,\colon\, \mathcal M^{Hod}_{M}(r)_{\mid \mathbb C^*}&\,\longrightarrow\, & 
\mathcal M^{Hod}_{\overline{M}}(r)_{\mid \mathbb C^*}\nonumber\\
(\lambda,\,\overline{\partial},\,D)_M&\,\longmapsto\,&
(\tfrac{1}{\lambda},\,\tfrac{1}{\lambda}D,\,\tfrac{1}{\lambda}\overline{\partial})_{\overline{M}}\, ;
\nonumber\end{eqnarray}
this covers the map $\mathbb C^*\,\longrightarrow\,\mathbb C^*$ defined by
$\lambda\,\longmapsto\,\tfrac{1}{\lambda}$.

The Deligne--Hitchin moduli space $$\mathcal M^{DH}_M(r)\,=\,\mathcal M^{Hod}_{ M}(r)\cup_\varphi {\mathcal
M}^{Hod}_{\overline{M}}(r)$$
is obtained by gluing the two Hodge moduli spaces via $\varphi$. It admits a fibration to ${\mathbb C}
{\mathbb P}^1$ which is $f_M$ on $\mathcal M^{Hod}_{ M}(r)$ and $1/f_{\overline{M}}$ on
$\mathcal M^{Hod}_{\overline{M}}(r)$. It should be mentioned that there does not exist a natural algebraic
structure on Deligne--Hitchin moduli spaces.

\subsection{Twistor space}

The twistor space of a hyperK\"ahler manifold $(N,\,g,\,I,\,J,\,K)$ encodes the hyperK\"ahler 
geometry in terms of complex analytic data. As a complex manifold it is given by $\mathcal 
P\,=\,N\times \mathbb C\mathbb P^1$ with (integrable) complex structure at 
$(p,\,\lambda)\,=(p,\,a+\sqrt{-1}b)$ given by
\[I_{(p,a+\sqrt{-1} b)}\,=\,(1-a\, K_p+b\,J_p)^{-1}I_p(1-a\, K_p+b\,J_p).\]
There is the holomorphic fibration $\mathcal P\,\longrightarrow\,\mathbb C\mathbb P^1,$ and
the manifold $N$ can be recovered as the space of {\em twistor lines}, i.e., as a component of holomorphic
sections of $\mathcal P\,\longrightarrow\,\mathbb C\mathbb P^1$ which are real with respect to the real structures
\[\rho\,\colon\, \mathcal P\,\longrightarrow\,\mathcal P\, ,\ \ (p,\,\lambda)\,\longmapsto\,
(p,\,-\overline{\lambda}^{-1})\]
$$
\mathbb C\mathbb P^1\, \longrightarrow\, \mathbb C\mathbb P^1,\, \ \ \lambda\,\longmapsto\, -\overline{\lambda}^{-1}
$$
of $\mathcal P$ and $\mathbb C\mathbb P^1$.

The various complex structures on $N$ are obtained by evaluating at specific $\lambda\,\in
\, \mathbb C\mathbb P^1.$ Moreover, the Riemannian metric
can be computed from a holomorphic twisted symplectic form on $\mathcal P$; see \cite{HKLR}.

As shown in \cite[Theorem 4.2]{Si5}, the Deligne--Hitchin moduli space is the twistor space of the hyperK\"ahler
space $\mathcal H_M(r)$.
The real structure $\rho$ on $\mathcal M^{DH}_M(r)$ is given by
\[\rho((\lambda,\,\overline{\partial},\,D)_M)\,=\,(-\overline{\lambda}^{-1},\,
\overline{\partial}^*,\,- D^*)_{\overline{M}}\, ,\]
where $\overline{\partial}^*$ and $D^*$ are the adjoint operators with respect to a Hermitian metric $h$ on the
underlying vector bundle $V$ (for the operators $\overline{\partial}$ and $D$).

Consider a polystable Higgs bundle $(E,\, \theta)$ on $M$ satisfying \eqref{e2}, and let $h$ 
be a Hermitian structure on $E$ that satisfies the Hermitian--Yang--Mills equation in 
\eqref{e3} for the Higgs bundle $(E,\, \theta)$. Let $(\nabla^h)^{0,1}$ be the holomorphic 
structure on $E$ determined by $E$. Then, the twistor line through the point 
$p\,=\,((\nabla^h)^{0,1},\,\theta)\,\in\, \mathcal H_M(r)$, over the open subset $\mathbb C\,\subset\,{\mathbb 
C}{\mathbb P}^1$, is given by
\[\lambda\,\longmapsto\, (\lambda,\,(\nabla^h)^{0,1}+\lambda\theta^*,\,\theta+\lambda (\nabla^h)^{1,0})\, ;\]
see \cite{Si5}. Note that this family interpolates between the Higgs pair $(E,\,\theta)$ at 
$\lambda\,=\,0$ and the corresponding flat connection \eqref{e4} at $\lambda\,=\,1$.

\subsection{Twistor subspaces}

Let
\begin{equation}\label{epi}
\pi\,\colon\,\mathcal P\,\longrightarrow\, \mathbb C\mathbb P^1
\end{equation}
be a twistor space of a 
hyperK\"ahler space $M.$ We call a complex subspace $\mathcal N\,\subset\,\mathcal P$ a 
twistor subspace if the following condition holds: whenever $s(\lambda_0)\,\in\, \mathcal N$ for a twistor line $s$ and some 
$\lambda_0\,\in\, \mathbb C\mathbb P^1$, then $s(\lambda)\, \in\, \mathcal N$ for all 
$\lambda\,\in\, {\mathbb C}{\mathbb P}^1.$

If $\mathcal N\,\subset\,\mathcal P$ is a twistor 
subspace then the restriction $\pi_{\mid \mathcal N}$ is a fibration and $$\rho(\mathcal N)\,=\,\mathcal N\, .$$ 
Moreover, the twisted holomorphic symplectic structure on $\mathcal P$ in \eqref{epi}
remains non-degenerate when restricted to the vertical tangent bundle 
of $\mathcal N.$ It follows that there exists a hyperK\"ahler subspace $N \,\subset\, M$ 
whose twistor space is given by $\mathcal N$. Conversely, every hyperK\"ahler subspace $N\, 
\subset\, M$ gives rise to a twistor subspace of $\mathcal P$ in \eqref{epi}.

The construction of the Deligne--Hitchin moduli space can be applied to the projective curve
$\phi\, :\, X\, \hookrightarrow\, M$ with induced surjective map
$$\phi_*\, :\, \pi_1(X,\, x_0)\, \longrightarrow\, \pi_1(M,\, x_0)$$ as in Section
\ref{Section:Restriction_to_curves}.
In that set-up, let
\begin{equation}\label{etwistz}
\Upsilon\, :\, {\mathcal M}^{DH}_M(r)\, \longrightarrow\, {\mathcal M}^{DH}_X(r)
\end{equation}
be the map defined by $(\lambda,\, E,\, D)\, \longmapsto\, (\lambda, \,\phi^*E,\, \phi^*D)$.
Theorem \ref{thm1} implies the following:

\begin{theorem}\label{thm2}
The map $\Upsilon$ in \eqref{etwistz} makes ${\mathcal M}^{DH}_M(r)$ a twistor subspace of
${\mathcal M}^{DH}_X(r)$.
\end{theorem}

\section{Real structures}\label{sec:real}

We shall dedicate this section to the study of real structures and their fixed point sets, 
and in particular, their appearance in the context of Higgs bundles.

\subsection{Real structures and Higgs bundles}

As before, $M$ is an irreducible smooth complex projective variety of dimension $d$. Let 
\begin{equation}\label{si}
\sigma\, :\, M\, \longrightarrow\, M
\end{equation}
be an anti-holomorphic involution of $M$. 

\begin{lemma}\label{lem12}
It is possible to choose $\phi$ in \eqref{e1} such that
\begin{enumerate}
\item $\sigma(\phi(X))\,=\, \phi(X)$, and 

\item the homomorphism $\phi_*\, :\, \pi_1(X,\, x_0)\, \longrightarrow\, \pi_1(M,\, x_0)$ 
is surjective.
\end{enumerate}\end{lemma}

\begin{proof}
Let $\mathcal L$ be a very ample line bundle on $M$. Then the holomorphic
line bundle $\sigma^*\overline{\mathcal L}$ on $M$ is also ample. The very ample line bundle
$$
{\mathbb L}\, :=\, {\mathcal L}\otimes \sigma^*\overline{\mathcal L}
$$
is equipped with a lift of the involution $\sigma$. This involution of $\mathbb L$,
which we shall denote by $\widetilde{\sigma}$, sends any $v_1\otimes v_2\, \in\,
{\mathbb L}_x\,=\, {\mathcal L}_x\otimes\overline{\mathcal L}_{\sigma(x)}$ to
$v_2\otimes v_1\, \in \, {\mathbb L}_{\sigma(x)}\,=\,
{\mathcal L}_{\sigma(x)}\otimes\overline{\mathcal L}_x$. Moreover, $\widetilde{\sigma}$
produces a conjugate linear involution
$$
\eta\, :\, H^0(M,\, {\mathbb L}) \, \longrightarrow\, H^0(M,\, {\mathbb L})
$$
that sends any $s\, \in \, H^0(M,\, {\mathbb L})$ to $\widetilde{\sigma}(s)$. The
fixed point set
$$
H^0(M,\, {\mathbb L})^\eta \, \subset\, H^0(M,\, {\mathbb L})
$$
is a totally real subspace, meaning $\dim_{\mathbb R} H^0(M,\, {\mathbb L})^\eta \, =
\,\dim_{\mathbb C} H^0(M,\, {\mathbb L})$, and
$$
H^0(M,\, {\mathbb L}) \,= \, H^0(M,\, {\mathbb L})^\eta\oplus 
\overline{H^0(M,\, {\mathbb L})^\eta}\, .
$$
Now the intersection of divisors of general $d-1$ elements of $H^0(M,\, {\mathbb L})^\eta$
is a smooth closed curve $Y$ in $M$ such that 
$\sigma(Y)\,=\, Y$, and the homomorphism $\pi_1(Y)\, \longrightarrow\, \pi_1(M)$ induced by the inclusion map
of $Y$ in $M$ is surjective, as required.
\end{proof}
 
Fix a smooth closed curve $X$ as in \eqref{e1}
such that
\begin{enumerate}
\item $\sigma(\phi(X))\,=\, \phi(X)$, and 

\item the homomorphism $\phi_*\, :\, \pi_1(X,\, x_0)\, \longrightarrow\, \pi_1(M,\, x_0)$
induced by $\phi$ in \eqref{e1} is surjective, where $x_0\, \in\, X$.
\end{enumerate}
We shall denote the restriction of the map $\sigma$ to $X$ by $\tau$.
The anti-holomorphic involution $\tau$ on $X$ produces a $C^\infty$ involution
\begin{equation}\label{ix}
{\mathcal I}_X\, :\, {\mathcal B}_X(r)\, \longrightarrow\, {\mathcal B}_X(r)
\end{equation}
of ${\mathcal B}_X(r)$ in \eqref{e5} that sends a flat connection $(E,\, \nabla)$ on $X$
to $(\tau^*E,\, \tau^*\nabla)$.
The anti-holomorphic involution $\sigma$ on $M$ produces a $C^\infty$ involution
\begin{equation}\label{im}
{\mathcal I}_M\, :\, {\mathcal B}_M(r)\, \longrightarrow\, {\mathcal B}_M(r)
\end{equation}
of ${\mathcal B}_M(r)$ in \eqref{e5} that sends a flat connection $(E,\, \nabla)$ on $X$
to $(\sigma^*E,\, \sigma^*\nabla)$.

From the above, the following is straightforward to prove.

\begin{theorem}\label{prop1}
For any $S\,\in\,\{X,\,M\}$, the involution ${\mathcal I}_S$ in $\{$\eqref{ix}, \eqref{im}$\}$ is
anti-holomorphic with respect to the
complex structure $I$ (that gives the complex structure of the moduli space
of Higgs bundles ${\mathcal H}_S(r)$).
The involution ${\mathcal I}_S$ is holomorphic with respect to the
complex structure $J$ (that gives the complex structure of
${\mathcal B}_S(r)$), and thus the fixed point set ${\mathcal H}_S(r)^{{\mathcal I}_S}$ is a
$(A,\,B,\,A)$--brane in the corresponding moduli space ${\mathcal H}_S(r)$.
Furthermore,
$$
{\mathcal I}_X\circ\Psi\,=\, \Psi\circ{\mathcal I}_M\, ,
$$
where $\Psi$ is the embedding in \eqref{e7}. 
\end{theorem}

\begin{remark}
When the variety $M$ is a Riemann surface, the involution ${\mathcal I}_M$ gives the map $f$ 
in \cite{aba} fixing an $(A,\,B,\,A)$--brane in the moduli space of Higgs bundles on $M$. 
Moreover, in such a situation ${\mathcal I}_M$ can naturally be seen as part of a triple of 
involutions fixing Lagrangian subspaces, as shown in \cite{slices}.
\end{remark}

\subsection{Real structures and $\lambda$--connections}

Let $X\,\hookrightarrow\, M$ be an embedded real curve as above, and let $\tau$ and $\sigma$ be the 
corresponding real structures. Note that for another irreducible smooth complex projective 
variety $N$ and an algebraic map $f\,\colon\, N \,\longrightarrow\, M$, we obtain an induced
map
$$
f^{M,N}\, :\, {\mathcal M}^{DH}_M(r)\, \longrightarrow\, {\mathcal M}^{DH}_N(r)
$$
generalizing $\Upsilon$ in \eqref{etwistz}. Thus, the commuting diagram 
$$
\begin{tikzcd}
X \ar[r, "\sim"] \ar[d, hookrightarrow] & \overline{X} \ar[d, hookrightarrow] \\
M \ar[r, "\sim"] & \overline{M} 
\end{tikzcd}
$$
yields for the twistor fibrations
\begin{equation}\label{eq:les2}
\begin{tikzcd}
{\mathcal M}^{DH}_M(r) \ar[r, "\sim"] \ar[d, hookrightarrow] & {\mathcal M}^{DH}_{\overline{M}}(r) \ar[d, hookrightarrow] \\
{\mathcal M}^{DH}_X(r) \ar[r, "\sim"] & {\mathcal M}^{DH}_{\overline{X}}(r).
\end{tikzcd}
\end{equation}
In particular, at $\lambda\,=\,1$ we obtain 
$$
\begin{tikzcd}
{\mathcal B}_M(r)\ar[rr, "I_M={f^{M,\overline{M}}}_{\mid\lambda=1}"] \ar[d,"\Phi" ]& &{\mathcal B}_{\overline{M}}(r) \ar[d, "\Phi"]\\
{\mathcal B}_X(r) \ar[rr, "I_X={f^{X,\overline{X}}}_{\mid\lambda=1}"]& &{\mathcal B}_{\overline{X}}(r)
\end{tikzcd}
$$
giving the $J$--part of Theorem \ref{prop1}. Recall that 
the fibers of a twistor space at $\lambda\,=\,0$ and $\lambda\,=\,\infty$ are complex conjugate spaces.
Hence, evaluating \eqref{eq:les2} at $\lambda\,=\,0$, we obtain the commutative diagram
$$
\begin{tikzcd}
{\mathcal H}_M(r)\ar[rr, "{f^{M,\overline{M}}}_{\mid\lambda=0}"] \ar[d,"\Psi" ] 
& &{\mathcal H}_{\overline{M}}(r) \ar[rr, "\sim"] \ar[d, "\Psi"] 
& & \overline{ {\mathcal H}_M(r)} \ar[d]\\
{\mathcal H}_X(r) \ar[rr, "{f^{X,\overline{X}}}_{\mid\lambda=1}"] & &{\mathcal H}_{\overline{X}}(r) \ar[rr, "\sim"] & & \overline{ {\mathcal H}_X(r)}
\end{tikzcd}
$$
which is in full accordance with the $I$--part of Theorem \ref{prop1}.

\section{Holomorphic action of a finite group}\label{sec:finite}

As in previous sections, let $M$ be an irreducible smooth complex projective variety of
complex dimension $d$, and consider a finite group $\Gamma$ acting
faithfully on $M$ via holomorphic automorphisms of $M$. For any $\gamma\, \in\, \Gamma$, the automorphism
of $M$ given by the action of $\gamma$ will also be denoted by $\gamma$.
For a very ample line bundle $L$ on $M$, the holomorphic line bundle
$$
{\mathcal L}\, :=\, \bigotimes_{\gamma\in\Gamma} \gamma^* L
$$
is both very ample and $\Gamma$--equivariant.

\subsection{Actions on higher dimensional varieties}

Consider the action of $\Gamma$ on $H^0(M,\, {\mathcal L})$ given by the action of
$\Gamma$ on $\mathcal L$. For general $d-1$ elements $s_1,\, \cdots ,\, s_{d-1}$ of
$H^0(M,\, {\mathcal L})^\Gamma$, the intersection
$$
X\, =\, \prod_{i=1}^{d-1} \text{divisor}(s_i)\, \subset\, M
$$
is a smooth projective curve satisfying the following two conditions:
\begin{itemize}
\item the homomorphism $\pi_1(X,\, x_0)\, \longrightarrow\, \pi_1(M,\, x_0)$ induced by the inclusion
map
$$
X\, \hookrightarrow\, M
$$
is surjective, where $x_0\, \in\, X$, and

\item the action of $\Gamma$ on $M$ preserves the curve $X$.
\end{itemize}

Fix a smooth closed curve $X$ as in \eqref{e1}
such that
\begin{enumerate}
\item the action of $\Gamma$ on $M$ preserves $\phi(X)$, and 

\item the homomorphism $\phi_*\, :\, \pi_1(X,\, x_0)\, \longrightarrow\, \pi_1(M,\, x_0)$ 
induced by $\phi$ in \eqref{e1} is surjective, where $x_0\, \in\, X$.
\end{enumerate}

\begin{lemma}
There is an embedding of the $\Gamma$--fixed locus ${\mathcal H}_M(r)^\Gamma$ in ${\mathcal H}_X(r)^\Gamma$ induced by the map $\Psi$ in \eqref{e7}. 
\end{lemma}
\begin{proof}
The action of $\Gamma$ on $M$ (respectively, $X$) produces an action of $\Gamma$ on the moduli space
${\mathcal H}_M(r)$ (respectively, ${\mathcal H}_X(r)$) of Higgs bundles on $M$ (respectively, $X$)
in \eqref{e7}; the action of any $\gamma\, \in\, \Gamma$ sends any $(E,\, \theta)$ to
$(\gamma^*E,\, \gamma^*\theta)$. The $\Gamma$--fixed locus in ${\mathcal H}_M(r)$ (respectively, ${\mathcal H}_X(r)$)
will be denoted by ${\mathcal H}_M(r)^\Gamma$ (respectively, ${\mathcal H}_X(r)^\Gamma$).
The map $\Psi$ in \eqref{e7} is evidently $\Gamma$--equivariant. Consequently, $\Gamma$ produces
a map
\begin{equation}\label{e7G}
\Psi^\Gamma\, :\, {\mathcal H}_M(r)^\Gamma\, \longrightarrow\, {\mathcal H}_X(r)^\Gamma\, .
\end{equation}
The map $\Psi^\Gamma$ in \eqref{e7G} is an embedding because $\Psi$ is so.
\end{proof}

\subsection{Actions on Riemann surfaces}

We shall now restrict our attention to the case where $M$ is a compact Riemann surface, and thus consider the 
moduli space $\mathcal{M}_{{\rm SL}(r,\mathbb{C})}$ of semistable Higgs bundles as introduced in \cite{Hi1} 
whose structure group is ${\rm SL}(r,\mathbb{C})$: that is, we consider $\mathcal{M}_{{\rm SL}(r,\mathbb{C})}
\,\subset\, {\mathcal H}_M(r)$. In this setting, we can see the following.

\begin{theorem}\label{teo1}
Assume that a finite group $\Gamma \, \subset\, {\rm Aut}(M)$ is acting on $M$ such
that the action is not free (so there are points with nontrivial isotropy).
Then, for $r\,>\,2$ the fixed point locus in $\mathcal{M}_{{\rm SL}(r,\mathbb{C})}$ for the
action of $\Gamma$ will never be a mid-dimensional space, and thus will
never be a Lagrangian subspace of $\mathcal{M}_{{\rm SL}(r,\mathbb{C})}$.
\end{theorem}

\begin{proof}
Assume that $r\,>\,2$. Then any connected component of the fixed point locus 
$$\mathcal{M}_{{\rm SL}(r,\mathbb{C})}^\Gamma\, \subset\,\mathcal{M}_{{\rm SL}(r,
\mathbb{C})}$$
for the action of $\Gamma$ is a component of a moduli space of parabolic
Higgs bundles of rank $r$ on the quotient surface $M/\Gamma$ of fixed determinant
for which the trace of the Higgs field is zero \cite{Bi,BMW}.

Consider the quotient map
$$
q\, :\, M\, \longrightarrow\, M/\Gamma\, =:\, Y\, .
$$
The genus of $Y$ will be denoted by $g_Y$. Let $x_1,\, \cdots, \, x_a\, \in\, Y$ be the 
points over which the map $q$ is ramified. For any $1\, \leq\, i\, \leq\, a$, let $r_i$ be the order 
of ramification of $q$ at any point of $q^{-1}(x_i)$ (clearly all points over $x_i$ have same 
order of ramification). By Hurwitz formula,
\begin{equation}\label{e1l}
2(g-1)\,=\, 2N\cdot (g_Y-1)+ \sum_{i=1}^a \frac{Nr_i}{r_i+1}\, ,
\end{equation}
where $N$ is the order of the group $\Gamma$.

We shall compute an upper bound for the dimension of the moduli space of ${\rm SL}(r,\mathbb{C})$ 
parabolic Higgs bundles with parabolic structure at the ramification points for $M$. The 
maximal possible dimension of quasi-parabolic filtrations at a point on a given bundle is 
$r(r-1)/2$, and the number of ramification points is $\sum_{i=1}^a \frac{N}{r_i+1}$. So an 
upper bound for the dimension of the moduli space of parabolic Higgs bundles of rank $r$ is
\begin{equation}\label{e2l}
B\, :=\, 2\left((r^2-1)(g_Y-1) + \frac{r(r-1)}{2}\sum_{i=1}^a \frac{N}{r_i+1}\right)\, ;
\end{equation}
note that the dimension of the moduli space of vector bundles of rank $r$ of fixed determinant
on the curve $Y$ is $(r^2-1)(g_Y-1)$. Now using \eqref{e1l} it follows immediately
that $B$ in \eqref{e2l} satisfies the inequality
$$
2B\, <\, 2(r^2-1)(g-1)\,=\, \dim \mathcal{M}_{{\rm SL}(r,\mathbb{C})}\, .
$$
The theorem follows from this inequality.
\end{proof}

\begin{remark}
It can be shown that the action of the group $\Gamma\,=:\,(\mathbb{Z}/2{\mathbb Z})
\times (\mathbb{Z}/2{\mathbb Z})$ considered in \cite{finite} does not give mid-dimensional
subspaces of the moduli spaces of higher rank Higgs bundles.
\end{remark}

\begin{proposition}
Assume that $M$ admits a holomorphic involution which is fixed point free.
Then the corresponding fixed point locus
$$\mathcal{M}_{{\rm SL}(r,\mathbb{C})}^{\mathbb{Z}/2{\mathbb Z}}
\, \subset\,\mathcal{M}_{{\rm SL}(r,\mathbb{C})}$$
is mid-dimensional for $r\, >\, 1$. \label{teo2}
\end{proposition}

\begin{proof}
Since $\Gamma\, =\, \mathbb{Z}/2\mathbb Z$ acts freely on $M$, the genus $g_Y$ of the quotient 
Riemann surface $Y\, :=\, M/\Gamma$ is
$$
g_Y\,=\, \frac{g+1}{2}\, .
$$
We note that the given condition that the involution of $M$ is fixed point free implies
that the genus $g$ of $M$ is odd.
Let $\mathcal{M}^Y_{{\rm SL}(r,\mathbb{C})}$ denote the moduli space of
semistable ${\rm SL}(r,\mathbb{C})$--Higgs bundles on $Y$. Then we have
$$\dim \mathcal{M}_{{\rm SL}(r,\mathbb{C})}\,=\, 2(r^2-1)(g-1)\,=\, 4(r^2-1)(g_Y-1)\,=\,
2\dim \mathcal{M}^Y_{SL(r,\mathbb{C})}\, .
$$
The proposition follows from this.
\end{proof}

From the above results, one can further understand the structure of the fixed point
locus as a brane of Higgs bundles. Indeed, in this setting one has the following. 

\begin{proposition} All the fixed point loci
$\mathcal{M}_{{\rm SL}(r,\mathbb{C})}^\Gamma$ constructed in Theorem \ref{teo1} and
Proposition \ref{teo2} are in fact $(B,\,B,\,B)$--branes.\label{teo3}
\end{proposition}

\begin{proof}
The fixed point locus 
$$
\mathcal{M}_{{\rm SL}(r,\mathbb{C})}^\Gamma\, \subset\,\mathcal{M}_{{\rm SL}(r,
\mathbb{C})}$$ is clearly holomorphic with respect to the natural holomorphic structure
on the moduli space $\mathcal{M}_{{\rm SL}(r,\mathbb{C})}$ of Higgs bundles. On the
other hand, the action of $\Gamma$ on $M$ produces a
homomorphism
$$
f_\Gamma\, :\, \Gamma\, \longrightarrow\, \text{Out}(\pi_1(M))
$$
to the group of outer automorphisms of the fundamental group of $M$. Using this
homomorphism $f_\Gamma$, the group $\Gamma$ acts on the character variety
$$
\text{Hom}(\pi_1(M),\, {\rm SL}(r,\mathbb{C}))/\!\!/ {\rm SL}(r,\mathbb{C})\, .
$$
This action of $\Gamma$ on $\text{Hom}(\pi_1(M),\, {\rm SL}(r,
\mathbb{C}))/\!\!/ {\rm SL}(r,\mathbb{C})$ is evidently given by algebraic automorphisms.

We note that if $\text{Hom}(\pi_1(M),\, {\rm SL}(r,\mathbb{C}))/\!\!/ {\rm SL}(r,\mathbb{C})$ is identified
with the moduli space of flat ${\rm SL}(r,\mathbb{C})$ connections on $M$, then the above action of
$\Gamma$ on $$\text{Hom}(\pi_1(M),\, {\rm SL}(r,\mathbb{C}))/\!\!/ {\rm SL}(r,\mathbb{C})$$ corresponds
to the following action of $\Gamma$ on the moduli space of flat ${\rm SL}(r,\mathbb{C})$ connections:
the action of any $\gamma\, \in\, \Gamma$ sends a flat connection $(F,\, \nabla)$ to the flat connection
$(\gamma^*F,\, \gamma^*\nabla)$.

The $C^\infty$ diffeomorphism
$$
\mathcal{M}_{{\rm SL}(r,\mathbb{C})}\, \stackrel{\sim}{\longrightarrow}\,
\text{Hom}(\pi_1(M),\, {\rm SL}(r,\mathbb{C}))/\!\!/ {\rm SL}(r,\mathbb{C})
$$
given by the nonabelian Hodge theory is $\Gamma$--equivariant.

{}From these we conclude that the action of $\Gamma$ on $\mathcal{M}_{{\rm 
SL}(r,\mathbb{C})}$ preserves all the complex structures in the family of
complex structures on
$$
\mathcal{M}_{{\rm SL}(r,\mathbb{C})}\, =\,
\text{Hom}(\pi_1(M),\, {\rm SL}(r,\mathbb{C}))/\!\!/ {\rm SL}(r,\mathbb{C})
$$
defining the hyperK\"ahler structure. Consequently, the fixed point locus
$\mathcal{M}_{{\rm SL}(r,\mathbb{C})}^\Gamma$ is a $(B,\,B,\,B)$--brane.
\end{proof}

\section*{Acknowledgements}

We thank the two referees for going through the paper very carefully.
IB is supported by a J. C. Bose Fellowship. LPS is partially supported by
NSF CAREER Award DMS-1749013. 

On behalf of all authors, the corresponding author states that there is no conflict of interest. 


\end{document}